\documentclass[12pt,amscd]{amsart}
\usepackage[all]{xy}
\usepackage{graphicx}
\usepackage{amsmath,amsxtra,amssymb,latexsym, amscd,amsthm}

\usepackage{tikz}
\usepackage{ytableau}

 \usepackage{indentfirst}
\usepackage[mathscr]{eucal}
  \usepackage[pagebackref=true]{hyperref}

\newtheorem{thm}{Theorem}[section]

\newtheorem{lem}[thm]{Lemma}

\theoremstyle{definition}
\newtheorem{defn}[thm]{Definition}

\numberwithin{equation}{section}

\DeclareMathOperator{\Ass}{Ass}

\begin{document}

\title{V-numbers of powers of cover ideals of unimodular hypergraphs}

\author[N.T. Hang]{Nguyen Thu Hang}
\address{Thai Nguyen University of Sciences, Phan Dinh phung Ward, Thai Nguyen, Vietnam}
\email{hangnt@tnus.edu.vn}

\author{Thanh Vu}
\address{Institute of Mathematics, VAST, 18 Hoang Quoc Viet, Hanoi, Vietnam}
\email{vuqthanh@gmail.com}


\date{}

\dedicatory{Dedicated to Professor Le Tuan Hoa on the occasion of his 70th birthday}
\commby{}

\begin{abstract} Let $H$ be a unimodular hypergraph with cover ideal $J(H)$. We prove that the local $v$-numbers of $J(H)^t$ are linear in $t$ for all $t\ge1$. We further show that the global $v$-number of $J(H)^t$ is linear in $t$ for all $t\ge n-1$. Finally, we prove that the global $v$-number of the powers of the cover ideal of any tree is linear in $t$ for all $t\ge1$.
\end{abstract}

\maketitle
\section{Introduction}
\label{sect_intro}
In \cite{V}, Vu proved that the local $v$-number of powers of the cover ideal of a bipartite graph is linear in $t$ for all $t\ge1$, and asked whether the global $v$-number is also linear in $t$ for all $t\ge1$. In this work, we generalize Vu's result by proving that the local $v$-number of powers of the cover ideal of any unimodular hypergraph is linear in $t$ for all $t\ge1$. We also answer Vu's question affirmatively for trees by showing that the global $v$-number of powers of the cover ideal of a tree is linear in $t$ for all $t\ge1$. The proof of the first result relies on the integer decomposition property of the covering polyhedron associated with a totally unimodular matrix. To establish the linearity of the global $v$-number for powers of cover ideals of trees, we analyze the covering numbers of subtrees obtained by deleting an edge. We now introduce the relevant notions in more detail.

A hypergraph $H$ is a pair $(V,E)$, where $V$ is a finite vertex set and $E$ is a collection of subsets of $V$, called edges. We assume throughout that no edge of $H$ contains another. Let $F_1,\ldots,F_m$ denote the edges of $H$, and let $M$ be the edge-vertex incidence matrix of $H$. We say that $H$ is a \emph{unimodular hypergraph} if $M$ is totally unimodular, that is, every square submatrix of $M$ has determinant $0$, $1$, or $-1$. Total unimodularity is a fundamental notion in combinatorial optimization and plays an important role in the study of integer linear programs \cite{S}. In particular, every bipartite graph is a unimodular hypergraph.

Let $S=\mathbf{k}[x_1,\ldots,x_n]$ be a standard graded polynomial ring over a field $\mathbf{k}$, and let $I$ be a nonzero homogeneous ideal of $S$. For an associated prime $P$ of $I$, the \emph{local $v$-number} of $I$ at $P$ is defined by
\[
v_P(I)
=
\min\left\{
d\ge0 \,\middle|\,
\text{there exists } f\in S_d \text{ such that } I:f=P
\right\}.
\]
The \emph{$v$-number} of $I$ is defined by
\[
v(I)=\min_{P\in\Ass(I)}v_P(I),
\]
where $\Ass(I)$ denotes the set of associated primes of $I$. The notion of the $v$-number of an ideal was introduced by Cooper, Seceleanu, Tohaneanu, Vaz Pinto, and Villarreal \cite{CSTVV}. Conca \cite{C} and, independently, Ficarra and Sgroi \cite{FS}, proved that the $v$-numbers of the powers of an ideal are eventually linear in the exponent.

When $H$ is a unimodular hypergraph, Herzog, Hibi, and Trung \cite{HHT} proved that
\[
J(H)^t=J(H)^{(t)}
\]
for all $t\ge1$. More recently, Chau, H\`a, Jayanthan, and Vu \cite{CHJV} showed that the $v$-numbers of symbolic powers of squarefree monomial ideals can be studied via integer programming. Combining this interpretation with the total unimodularity of the incidence matrix, we prove that the local $v$-numbers of powers of the cover ideal of a unimodular hypergraph are linear in $t$ for all $t\ge1$. Our first main result is the following.

\begin{thm}\label{thm_unimodular}
Let $H$ be a unimodular hypergraph, let $P$ be an edge of $H$, and write
\(
P=(x_j\mid j\in P)
\)
for the corresponding prime ideal of $S$. Then, for all $t\ge1$,
\[
v_P\bigl(J(H)^t\bigr)
=
v_P\bigl(J(H)\bigr)+(t-1)\gamma_P
\]
for some integer constant $\gamma_P$. Consequently,
\(
v\bigl(J(H)^t\bigr)
\)
is a linear function of $t$ for all $t\ge n-1$.
\end{thm}

The proof of Theorem~\ref{thm_unimodular} is given in the next section. We next address a question of Vu \cite{V} concerning the linearity of the $v$-number of powers of cover ideals of bipartite graphs. By analyzing the covering numbers of certain subtrees obtained by deleting an edge, we prove the following second main result.

\begin{thm}\label{thm_linear_trees}
Let $T$ be a tree. Then $v\bigl(J(T)^t\bigr)$ is a linear function of $t$ for all $t\ge1$.
\end{thm}

\section{$v$-number of symbolic powers of cover ideals}

Throughout this paper, let $S = k[x_1, \ldots, x_n]$ be a standard graded polynomial ring over a field $k$. For a non-zero homogeneous ideal $I \subseteq S$, we denote by $\operatorname{Ass}(I)$ the set of associated primes of $I$. Several works address the asymptotic behavior of the $v$-function of symbolic powers \cite{KNS,VS}, whereas its precise behavior for small powers was partially studied in \cite{CHJV}.

\subsection{Local $v$-numbers via integer programming}

For a monomial $f$, we denote by $\deg(f)$ its total degree and, for each $i \in [n]$, by $\deg_i(f)$ the exponent of $x_i$ in $f$. For a monomial prime ideal $P$, we further define
\[
\deg_P(f) = \sum_{x_i \in P} \deg_i(f).
\]

The following result, which follows from \cite{CHJV}, enables the computation of the local $v$-numbers of symbolic powers of cover ideals of a hypergraph $H$ via integer programming.

\begin{lem}\label{lem_criterion}
Let $H$ be a hypergraph with edges $F_1, \ldots, F_m$. For an associated prime $P = (x_j \mid j \in F_i)$, we have 
\[
v_P(J(H)^{(t)}) = \min \left\{ \sum_{j=1}^n a_j \;\middle|\; \mathbf{a} \in \mathbb{N}^n, \, \sum_{j \in F_i} a_j = t-1, \, \sum_{j \in F_\ell} a_j \ge t \text{ for all } \ell \neq i \right\}.
\]
\end{lem}

\begin{proof}
By \cite[Lemma 2.3]{CHJV}, a monomial $f$ satisfies $(J(H)^{(t)} : f) = P$ if and only if $\deg_P(f) = t-1$ and $\deg_Q(f) \ge t$ for all other associated primes $Q$ of $J(H)$. The conclusion then follows directly from the definition.
\end{proof}

\subsection{Hypergraphs and their cover ideals}

Let $H$ be a hypergraph with vertex set $V(H) = \{1, \ldots, n\}$ and edges $F_1, \ldots, F_m$. The \emph{cover ideal} of $H$, denoted by $J(H)$, is defined by
\[
J(H) = \bigcap_{i=1}^m (x_j \mid j \in F_i).
\]
We denote by $\tau(H)$ the vertex covering number of $H$, i.e., the minimum cardinality of a vertex cover of $H$.

\begin{defn}
Let $H$ be a hypergraph with edges $F_1, \ldots, F_m$. A vector $\mathbf{a} \in \mathbb{N}^n$ is called a $t$-cover of $H$ if 
\[
\sum_{j \in F_i} a_j \ge t \quad \text{for all } i = 1, \ldots, m.
\]
\end{defn}

When $H$ is unimodular, we have the following result \cite[Theorem 1.1]{HHT}:

\begin{thm}\label{thm_decomp}
Let $H$ be a unimodular hypergraph and let $\mathbf{a} \in \mathbb{N}^n$ be a $t$-cover of $H$. Then there exist $t$ $1$-covers $\mathbf{a}_1, \ldots, \mathbf{a}_t$ of $H$ such that $\mathbf{a} = \mathbf{a}_1 + \cdots + \mathbf{a}_t$. 
\end{thm}

\begin{defn}
Let $H$ be a hypergraph with edges $F_1, \ldots, F_m$. Fix an edge $P = F_i$ in $H$. A vector $\mathbf{a} \in \mathbb{N}^n$ is called a $P$-almost $t$-cover of $H$ if 
\[
\sum_{j \in F_i} a_j = t-1 \quad \text{and} \quad \sum_{j \in F_\ell} a_j \ge t \quad \text{for all } \ell \neq i.
\]
\end{defn}

We now prove the following decomposition lemma for $P$-almost $t$-covers of unimodular hypergraphs.

\begin{lem}\label{lem_decomp1}
Let $H$ be a unimodular hypergraph, and let $P = F_i$ be an edge of $H$. Let $\mathbf{c} \in \mathbb{N}^n$ be a $P$-almost $(t+1)$-cover of $H$. Then there exist a $1$-cover $\mathbf{b}$ and a $P$-almost $t$-cover $\mathbf{a}$ of $H$ such that $\mathbf{c} = \mathbf{a} + \mathbf{b}$.
\end{lem}

\begin{proof}
We consider the hypergraph $\Delta$ with vertex set $V(\Delta) = V(H) \cup \{n+1\}$ and edge set $E(\Delta) = \{F_1, \ldots, F_i', \dots, F_m\}$, where $F_i' = F_i \cup \{n+1\}$. The hypergraph $\Delta$ is unimodular because its incidence matrix $M_\Delta$ is obtained from $M_H$ by appending a standard basis column vector $e_{i}$ containing a single $1$ at the $i$-th position and $0$s elsewhere. Every square submatrix of $M_\Delta$ is either a submatrix of $M_H$ or, after permuting rows and columns if necessary, a block upper triangular matrix consisting of a square submatrix of $M_H$ and a $1 \times 1$ block containing $1$. Thus, $M_\Delta$ remains totally unimodular.

Now, let $\mathbf{d} \in \mathbb{N}^{n+1}$ be the vector $(\mathbf{c}, 1)$. Then $\mathbf{d}$ is a $(t+1)$-cover of $\Delta$. By Theorem~\ref{thm_decomp}, there exist $t+1$ $1$-covers $\mathbf{d}_1, \ldots, \mathbf{d}_{t+1}$ of $\Delta$ such that $\mathbf{d} = \mathbf{d}_1 + \cdots + \mathbf{d}_{t+1}$. 

Since $\sum_{j \in F_i'} d_j = t+1$ and each $\mathbf{d}_\ell$ is a $1$-cover of $\Delta$, we must have $\sum_{j \in F_i'} (\mathbf{d}_\ell)_j = 1$ for all $\ell = 1, \ldots, t+1$. Because $d_{n+1} = 1$, exactly one vector $\mathbf{d}_\ell$ has its $(n+1)$-th coordinate equal to $1$, while all other $\mathbf{d}_\ell$ have $0$ as their $(n+1)$-th coordinate. Without loss of generality, assume that $(\mathbf{d}_{t+1})_{n+1} = 1$.

For each $\ell = 1, \ldots, t+1$, let $\mathbf{a}_\ell \in \mathbb{N}^n$ be the restriction of $\mathbf{d}_\ell$ obtained by dropping the $(n+1)$-th coordinate. Then $\mathbf{a}_1, \ldots, \mathbf{a}_t$ are $1$-covers of $H$, and $\mathbf{a}_{t+1}$ is a $P$-almost $1$-cover of $H$. Setting $\mathbf{b} = \mathbf{a}_1$ and $\mathbf{a} = \mathbf{a}_2 + \cdots + \mathbf{a}_{t+1}$, we obtain the desired decomposition $\mathbf{c} = \mathbf{a} + \mathbf{b}$.
\end{proof}

\begin{defn}
Let $H$ be a hypergraph with edges $F_1, \ldots, F_m$. Fix an edge $P = F_i$ in $H$. A vector $\mathbf{a} \in \mathbb{N}^n$ is called a $P$-tight $t$-cover of $H$ if 
\[
\sum_{j \in F_i} a_j = t \quad \text{and} \quad \sum_{j \in F_\ell} a_j \ge t \quad \text{for all } \ell \neq i.
\]
We define
\[
\gamma_{t,P} = \min \{ |\mathbf{a}| \mid \mathbf{a} \text{ is a } P\text{-tight } t\text{-cover of } H \}.
\]
\end{defn}

\begin{proof}[Proof of Theorem~\ref{thm_unimodular}]
Since $H$ is a unimodular hypergraph, by \cite{HHT} we have $J(H)^t = J(H)^{(t)}$ for all $t \ge 1$. Let $\mathbf{b}$ be a $P$-tight $1$-cover of $H$ such that $|\mathbf{b}| = \gamma_{1,P}$. By Lemma~\ref{lem_criterion}, there exists a $P$-almost $t$-cover $\mathbf{a}$ of $H$ such that $v_P(J(H)^t) = |\mathbf{a}|$. The sum $\mathbf{a} + \mathbf{b}$ is then a $P$-almost $(t+1)$-cover of $H$. Thus, by Lemma~\ref{lem_criterion},
\[
v_P(J(H)^{t+1}) \le |\mathbf{a}| + |\mathbf{b}| = v_P(J(H)^t) + \gamma_{1,P}.
\]

Conversely, let $\mathbf{c}$ be a $P$-almost $(t+1)$-cover of $H$ such that $|\mathbf{c}| = v_P(J(H)^{t+1})$. By Lemma~\ref{lem_decomp1}, we can write $\mathbf{c} = \mathbf{a} + \mathbf{b}$, where $\mathbf{b}$ is a $P$-tight $1$-cover of $H$ and $\mathbf{a}$ is a $P$-almost $t$-cover of $H$. By Lemma~\ref{lem_criterion},
\[
v_P(J(H)^{t+1}) = |\mathbf{a}| + |\mathbf{b}| \ge v_P(J(H)^t) + \gamma_{1,P}.
\]
Hence, $v_P(J(H)^{t+1}) - v_P(J(H)^t) = \gamma_{1,P}$ for all $t \ge 1$.

For each prime ideal $P$ corresponding to an edge of $H$, set $f_P(t) = v_P(J(H)^t)$. Note that $\tau(H) = \min_{P} \gamma_{1,P}$. Let $Q$ be an associated prime ideal such that $\gamma_{1,Q} = \tau(H)$ and $f_Q(1)$ is as small as possible among all such choices of $Q$.

Now, for any other associated prime $P$ where $\gamma_{1,P} > \gamma_{1,Q}$, we have
\[
f_P(t) - f_Q(t) = (f_P(1) - f_Q(1)) + (t-1)(\gamma_{1,P} - \gamma_{1,Q}).
\]
For the edge $Q$, set $a_j = 0$ for all $j \in Q$ and $a_j = 1$ for all $j \notin Q$. Then $\mathbf{a}$ is a $Q$-almost $1$-cover of $H$. Hence, $f_Q(1) \le |\mathbf{a}| \le n-1$. Furthermore, since $f_P(1) \ge 1$, we obtain $f_P(1) - f_Q(1) \ge 1 - (n-1) = 2 - n$. Since $\gamma_{1,P} - \gamma_{1,Q} \ge 1$, this difference is non-negative for all $t \ge n-1$. Consequently, for all $t \ge n-1$, $v(J(H)^t) = f_Q(t)$. The conclusion follows.
\end{proof}

\section{Global $v$-number of powers of cover ideals of trees}
In this section, we prove that the $v$-number of powers of cover ideals of trees is linear for all $t \ge 1$. We first recall the following standard definitions.

\begin{defn}\label{definition1}
Let $G$ be a simple graph with vertex set $V(G) = \{1, \ldots, n\}$ and edge set $E(G)$.
\begin{enumerate}
    \item A simple graph $H$ is a \emph{subgraph} of $G$ if $V(H) \subseteq V(G)$ and $E(H) \subseteq E(G)$. It is an \emph{induced subgraph} of $G$ if $E(H) = \{ \{u, v\} \in E(G) \mid u, v \in V(H) \}$.
    
    \item For a subset $U \subseteq V(G)$, we denote by $G[U]$ and $G \setminus U$ the induced subgraphs of $G$ with vertex sets $U$ and $V(G) \setminus U$, respectively.
    
    \item For a vertex $v \in V(G)$, we denote by $N_G(v) = \{u \in V(G) \mid \{u, v\} \in E(G)\}$ the set of neighbors of $v$. For a subset $U \subseteq V(G)$, $N_G(U) = \bigcup_{u \in U} N_G(u)$ denotes the open neighborhood of $U$.
    
    \item A \emph{cycle} $C_n$ on $n$ vertices is a graph with vertex set $V(C_n) = \{1, \ldots, n\}$ and edge set
    \[
    E(C_n) = \{ \{1,2\}, \{2,3\}, \ldots, \{n-1,n\}, \{1,n\} \}.
    \]
    
    \item A \emph{forest} is an acyclic graph. A \emph{tree} is a connected forest.
\end{enumerate}
\end{defn}
By convention, we set $\tau(G) = 0$ if $G$ has no edges. Let $T$ be a tree. By \cite[Theorem 1.3]{V}, to establish the linearity of the global $v$-number of powers $J(T)^t$ for all $t \ge 1$, it suffices to find an edge whose corresponding prime ideal minimizes both the constant term and the leading coefficient of the local $v$-function. Since the leading coefficient of the global linear function coincides with the covering number $\tau(T)$, we analyze this relationship by decomposing $T$ along its edges.

For each edge $\{u,v\}$ of $T$, removing $\{u,v\}$ disconnects $T$ into exactly two connected components: $T_{u \to v}$ containing $u$, and $T_{v \to u}$ containing $v$. The conditions for $P$-tight and $P$-almost covers reduce to covering problems on these two subtrees, allowing us to express the local $v$-number invariants in terms of subtree data and prove global linearity.

\begin{lem}\label{lem:2}
Let $T$ be a tree, and let $\{u,v\}$ be an edge of $T$. Then
\[
T_{u\to v}-u
=
\bigcup_{w\in N_T(u)\setminus\{v\}} T_{w\to u}.
\]
\end{lem}

\begin{proof}
After removing the edge $\{u,v\}$, the component $T_{u\to v}$ contains
$u$ and every neighbor of $u$ other than $v$. Since $T$ is a tree, every
vertex of $T_{u\to v}\setminus\{u\}$ is connected to $u$ through a
unique neighbor $w\in N_T(u)\setminus\{v\}$. Thus, after removing $u$
from $T_{u\to v}$, the resulting connected components are precisely
the subtrees $T_{w\to u}$ for
$w\in N_T(u)\setminus\{v\}$. The conclusion follows.
\end{proof}

\begin{defn}
Let $T$ be a tree and let $\{u,v\}$ be an edge of $T$. We define 
\[
\alpha_{u \to v} = \min \{ |C| \mid C \text{ is a vertex cover of } T_{u \to v} \text{ with } u \notin C \},
\]
\[
\beta_{u \to v} = \min \{ |C| \mid C \text{ is a vertex cover of } T_{u \to v} \text{ with } u \in C \},
\]
and $d_{u \to v} = \beta_{u \to v} - \alpha_{u \to v}$.
\end{defn}

\begin{lem} \label{lem:3} 
Let $T$ be a tree and let $\{u,v\}$ be an edge of $T$. Then 
\[
d_{u \to v} = 1 - \sum_{w \in N_T(u) \setminus \{v\}} \max \{ 0, d_{w \to u} \}.
\]
\end{lem}

\begin{proof} 
Let $C$ be a vertex cover of $T_{u \to v}$ with $u \notin C$. Then $w \in C$ for all $w \in N_T(u) \setminus \{v\}$. By Lemma~\ref{lem:2}, we deduce that 
\[
\alpha_{u \to v} = \sum_{w \in N_T(u) \setminus \{v\}} \beta_{w \to u}.
\]

Now assume that $C$ is a vertex cover of $T_{u \to v}$ with $u \in C$. By Lemma~\ref{lem:2}, $C$ must cover $T_{w \to u}$ for each $w \in N_T(u) \setminus \{v\}$. Hence, 
\[
\beta_{u \to v} = 1 + \sum_{w \in N_T(u) \setminus \{v\}} \min \{ \alpha_{w \to u}, \beta_{w \to u} \}.
\]
Therefore, 
\begin{align*}
d_{u \to v} &= \beta_{u \to v} - \alpha_{u \to v} \\
&= 1 + \sum_{w \in N_T(u) \setminus \{v\}} (\min \{ \alpha_{w \to u}, \beta_{w \to u} \} - \beta_{w \to u}) \\
&= 1 - \sum_{w \in N_T(u) \setminus \{v\}} \max \{ 0, d_{w \to u} \}.
\end{align*}
The conclusion follows.
\end{proof}

\begin{defn}
Let $G$ be a simple graph with cover ideal $J(G)$. For an edge $\{u,v\}$ of $G$ with corresponding prime ideal $P = (x_u, x_v)$ of $S$, we set 
\[
f_{1,uv} = v_P(J(G)) \quad \text{and} \quad \gamma_{1,uv} = \gamma_{1,P}.
\]
\end{defn}

\begin{lem}\label{lem_tau} 
Let $G$ be a simple graph. Then 
\[
\tau(G) = \min_{P} \gamma_{1,P},
\]
where $P$ ranges over all prime ideals $P = (x_u, x_v)$ corresponding to edges $\{u,v\} \in E(G)$.
\end{lem}

\begin{proof}
Since $\gamma_{1,P}$ is the minimum weight of a $P$-tight $1$-cover of $G$, we clearly have $\tau(G) \le \gamma_{1,P}$ for all such $P$. 

Conversely, let $\mathbf{a}$ be a $1$-cover of $G$ of minimum weight $\tau(G)$. Then $\mathbf{a}$ must be $P$-tight for at least one prime ideal $P = (x_u, x_v)$ associated with an edge $\{u,v\} \in E(G)$. Indeed, if $a_u + a_v \ge 2$ for every edge $\{u,v\}$, then decrementing any nonzero entry of $\mathbf{a}$ by $1$ would yield a valid $1$-cover of strictly smaller weight, contradicting the minimality of $\tau(G)$. Hence, $\tau(G) \ge \gamma_{1,P}$ for some $P$, completing the proof.
\end{proof}

We now analyze $f_{1,uv}$ and $\gamma_{1,uv}$ in terms of invariants of the subtrees $T_{u \to v}$ and $T_{v \to u}$.

\begin{lem}\label{lem_4} 
Let $T$ be a tree and let $\{u,v\}$ be an edge of $T$. Then 
\[
f_{1,uv} = \alpha_{u \to v} + \alpha_{v \to u}.
\]
\end{lem}

\begin{proof} 
Let $P = (x_u, x_v)$ be the corresponding prime ideal. By Lemma~\ref{lem_criterion}, $f_{1,uv}$ is the minimum weight of a $P$-almost $1$-cover $\mathbf{a}$ of $T$. In particular, $a_u = 0$ and $a_v = 0$, and $\mathbf{a}$ covers all other edges of $T$. Thus, $\mathbf{a}$ must cover $T_{u \to v}$ and $T_{v \to u}$ independently with $u, v \notin \text{supp}(\mathbf{a})$. The conclusion follows.
\end{proof}

\begin{lem}\label{lem_5} 
Let $T$ be a tree and let $\{u,v\}$ be an edge of $T$. Then 
\[
\gamma_{1,uv} = \min \{\alpha_{u \to v} + \beta_{v \to u}, \beta_{u \to v} + \alpha_{v \to u}\}.
\]
\end{lem}

\begin{proof} 
By definition, $\gamma_{1,uv}$ is the minimum weight of a $P$-tight $1$-cover of $T$, where $P = (x_u, x_v)$. Let $\mathbf{a}$ be a $P$-tight $1$-cover of $T$. Then $a_u + a_v = 1$, so either $a_u = 1$ and $a_v = 0$, or $a_u = 0$ and $a_v = 1$. 

In the first case, $\mathbf{a}$ restricts to a vertex cover $C$ of $T_{u \to v}$ with $u \in C$ and a vertex cover $D$ of $T_{v \to u}$ with $v \notin D$. Thus, $|\mathbf{a}| \ge \beta_{u \to v} + \alpha_{v \to u}$. In the second case, an analogous argument yields $|\mathbf{a}| \ge \alpha_{u \to v} + \beta_{v \to u}$. Consequently, the left-hand side is at least the right-hand side.

Conversely, the union of a vertex cover $C$ of $T_{u \to v}$ with $u \in C$ and a vertex cover $D$ of $T_{v \to u}$ with $v \notin D$ (or vice versa) forms a valid $P$-tight $1$-cover of $T$. Therefore, the reverse inequality holds, and equality follows.
\end{proof}

\begin{lem}\label{lem_6} 
Let $T$ be a tree and let $\{u,v\}$ be an edge of $T$. Then 
\[
\tau(T) = \min \{\alpha_{u \to v} + \beta_{v \to u}, \beta_{u \to v} + \alpha_{v \to u}, \beta_{u \to v} + \beta_{v \to u}\}.
\]
\end{lem}

\begin{proof} 
Let $\mathbf{a}$ be a minimal $1$-cover of $T$. Since $\{u,v\}$ is an edge of $T$, we must have $a_u + a_v \ge 1$. This yields three disjoint possibilities for the pair $(a_u, a_v)$:
\begin{enumerate}
    \item $a_u \ge 1$ and $a_v = 0$: In this case, $\mathbf{a}$ restricts to a vertex cover of $T_{u \to v}$ containing $u$ and a vertex cover of $T_{v \to u}$ not containing $v$. Thus, $|\mathbf{a}| \ge \beta_{u \to v} + \alpha_{v \to u}$.
    \item $a_u = 0$ and $a_v \ge 1$: Analogously, $|\mathbf{a}| \ge \alpha_{u \to v} + \beta_{v \to u}$.
    \item $a_u \ge 1$ and $a_v \ge 1$: Here, $\mathbf{a}$ restricts to vertex covers containing $u$ and $v$ in $T_{u \to v}$ and $T_{v \to u}$, respectively, giving $|\mathbf{a}| \ge \beta_{u \to v} + \beta_{v \to u}$.
\end{enumerate}
Taking the minimum over all three valid configurations yields the desired equality.
\end{proof}

\begin{lem}\label{lem_7} 
Let $T$ be a tree and let $\{u,x\}$ and $\{u,y\}$ be edges of $T$. Then 
\[
f_{1,ux} - f_{1,uy} = d_{y \to u} - d_{x \to u}.
\]
\end{lem}

\begin{proof} 
As established in the proof of Lemma~\ref{lem:3}, for any neighbor $w \in N_T(u)$, we have
\[
\alpha_{u \to w} = \sum_{z \in N_T(u) \setminus \{w\}} \beta_{z \to u}.
\]
Applying this to $w = x$ and $w = y$, we obtain
\[
\alpha_{u \to x} - \alpha_{u \to y} = \sum_{z \in N_T(u) \setminus \{x\}} \beta_{z \to u} - \sum_{z \in N_T(u) \setminus \{y\}} \beta_{z \to u} = \beta_{y \to u} - \beta_{x \to u}.
\]
By Lemma~\ref{lem_4}, $f_{1,ux} = \alpha_{u \to x} + \alpha_{x \to u}$ and $f_{1,uy} = \alpha_{u \to y} + \alpha_{y \to u}$. Therefore,
\begin{align*}
f_{1,ux} - f_{1,uy} &= (\alpha_{u \to x} + \alpha_{x \to u}) - (\alpha_{u \to y} + \alpha_{y \to u}) \\
&= (\alpha_{u \to x} - \alpha_{u \to y}) + \alpha_{x \to u} - \alpha_{y \to u} \\
&= (\beta_{y \to u} - \beta_{x \to u}) + \alpha_{x \to u} - \alpha_{y \to u} \\
&= (\beta_{y \to u} - \alpha_{y \to u}) - (\beta_{x \to u} - \alpha_{x \to u}) \\
&= d_{y \to u} - d_{x \to u}.
\end{align*}
The conclusion follows.
\end{proof}

\begin{proof}[Proof of Theorem \ref{thm_linear_trees}] 
By \cite[Theorem 1.3]{V}, for each edge $\{u,v\}$ with corresponding prime ideal $P = (x_u, x_v)$, the local $v$-number is 
\[
v_P(J(T)^t) = f_{1,uv} + (t-1) \gamma_{1,uv}.
\]

Let $\{u,v\}$ be an edge of $T$ that minimizes $f_{1,uv}$. To establish linearity, it suffices to prove that $\gamma_{1,uv} = \tau(T)$. We first claim that $d_{u \to v}$ and $d_{v \to u}$ cannot both be negative. Assume for contradiction that $d_{u \to v} < 0$ and $d_{v \to u} < 0$. 

By Lemma~\ref{lem:3}, we have 
\[
d_{u \to v} = 1 - \sum_{w \in N_T(u) \setminus \{v\}} \max \{ 0, d_{w \to u} \}.
\]
In particular, $d_{x \to y} \le 1$ for all edges $\{x,y\}$ in $T$. Furthermore, since $d_{u \to v} < 0$, there must exist some $w \in N_T(u) \setminus \{v\}$ such that $d_{w \to u} = 1$. 

By Lemma~\ref{lem_7}, we then have 
\[
f_{1,uw} - f_{1,uv} = d_{v \to u} - d_{w \to u} = d_{v \to u} - 1 < 0,
\]
which yields $f_{1,uw} < f_{1,uv}$. This contradicts the minimality of $f_{1,uv}$. Thus, at least one of $d_{u \to v}$ or $d_{v \to u}$ is nonnegative.

Now, Lemma~\ref{lem_5} yields 
\begin{align*}
\gamma_{1,uv} &= \min \{ \alpha_{u \to v} + \beta_{v \to u}, \beta_{u \to v} + \alpha_{v \to u} \} \\
&= f_{1,uv} + \min \{ d_{u \to v}, d_{v \to u} \},
\end{align*}
and Lemma~\ref{lem_6} yields 
\begin{align*}
\tau(T) &= \min \{ \alpha_{u \to v} + \beta_{v \to u}, \beta_{u \to v} + \alpha_{v \to u}, \beta_{u \to v} + \beta_{v \to u} \} \\
&= f_{1,uv} + \min \{ d_{u \to v}, d_{v \to u}, d_{u \to v} + d_{v \to u} \}.
\end{align*}
Since $\max \{ d_{u \to v}, d_{v \to u} \} \ge 0$, we have $d_{u \to v} + d_{v \to u} \ge \min \{ d_{u \to v}, d_{v \to u} \}$, which implies 
\[
\min \{ d_{u \to v}, d_{v \to u}, d_{u \to v} + d_{v \to u} \} = \min \{ d_{u \to v}, d_{v \to u} \}.
\]
Consequently, $\gamma_{1,uv} = \tau(T)$, completing the proof.
\end{proof}

\vspace{0.2cm}
\noindent {\bf Data Availability} Data sharing is not applicable to this article as no datasets were generated or analyzed during the current study.
\vspace{0.2cm}

\noindent {\bf Conflict of interest} There are no competing interests of either financial or personal nature.

\end{document}